\documentclass[11pt,reqno,a4paper]{amsart}
\usepackage{graphicx,color}

\usepackage{amsmath,amsfonts,amssymb,amsthm,mathrsfs,amsopn}
\usepackage{latexsym}

\usepackage[numbers]{natbib}

\usepackage{hyperref}

\usepackage[usenames,dvipsnames,x11names,table]{xcolor}

\numberwithin{equation}{section}

\def\H{\mathcal H}

\def\R{\mathbb R}
\def\N{\mathbb N}

\newcommand{\dist}{\mathop{\mathrm{dist}}}
\def\e{\varepsilon}

\def\vphi{\varphi}

\def\Id{{\rm Id}}

\def\spt{{\rm spt}}

\def\pa{\partial}

\def\00{{\bf 0}}

\def\P{\mathcal{P}}

\renewcommand{\b}{\beta}

\newcommand{\D}{\Delta}

\renewcommand{\L}{\Lambda}

\newcommand{\tr}{\mbox{tr }}

\def\Lip{{\rm Lip}\,}

\newcommand\res{\mathop{\hbox{\vrule height 7pt width .3pt depth 0pt \vrule height .3pt width 5pt depth 0pt}}\nolimits}

\def\FF{\mathbf{F}}

\def\D{\mathsf{D}}

\def\Id{\mathrm{Id}}

\theoremstyle{plain}
\newtheorem{theorem}{Theorem} [section]

\newtheorem{lemma}[theorem]{Lemma}

\theoremstyle{definition}
\newtheorem{definition}[theorem]{Definition}
\newtheorem{remark}[theorem]{Remark}
\newtheorem*{ack}{Acknowledgements}

\title{Minimization of anisotropic energies in classes of rectifiable varifolds}

\author{Antonio De Rosa}
\address{Institut f\"ur Mathematik, Universitaet Z\"urich, Winterthurerstrasse 190, CH-8057 Z\"urich, Switzerland}
\email{antonio.derosa@math.uzh.ch}

\begin{document}

\begin{abstract} 
We consider the minimization problem of an anisotropic energy in classes of $d$-rectifiable varifolds in $\R^n$, closed under Lipschitz deformations and encoding a suitable notion of boundary. We prove that any minimizing sequence with density uniformly bounded from below converges (up to subsequences) to a $d$-rectifiable varifold. Moreover, the limiting varifold is integral, provided the minimizing sequence is made of integral varifolds with uniformly locally bounded anisotropic first variation.
\end{abstract}

\maketitle

\section{Introduction}\label{intro}

In the recent paper \cite{DeLGhiMag}, De Lellis, Ghiraldin and Maggi propose a direct approach to the minimization of the Hausdorff measure in certain classes of sets of codimension one, which has proven to be fairly general to solve different formulations of the Plateau problem. Their result has been extended in \cite{DePDeRGhi} by De Philippis, Ghiraldin and the author to the general codimension case  and in \cite{DeLDeRGhi16} by De Lellis, Ghiraldin and the author to the anisotropic setting (but in codimension one). 

One of the main issues in these compactness results was the rectifiability of the minimizing set. For the area functional, an important rectifiability result of Preiss \cite{Preiss} (see also \cite{DeLellisNOTES}) and the powerful monotonicity formula were used to solve this task.
In the anisotropic case, the lack of monotonicity formulas has imposed to find a different strategy in \cite{DeLDeRGhi16}, which however is not applicable in the general codimension setting, because it is based on the theory of Caccioppoli sets.

De Philippis, Ghiraldin and the author proved in \cite{DePDeRGhi2} the anisotropic counterpart of the Allard's rectifiability theorem, \cite{Allard}, for varifolds with bounded first variation with respect to an anisotropic integrand. This tool will be applied by the same authors in \cite{DePDeRGhi3} to extend the solutions of the Plateau problem in \cite{DeLGhiMag,DePDeRGhi,DeLDeRGhi16} to the anisotropic functionals in general codimension.

As in the case of the area integrand, 
\cite{DeGiorgiSOFP1,federer60,reifenberg1,Allard,AlmgrenINVITATION,DavidSemmes,
HarrisonPugh14,DeLGhiMag,DePDeRGhi}, many definitions of boundary conditions (both homological and homotopical), 
as well as the type of competitors (currents, varifolds, sets) have been considered in the literature for the minimization of elliptic integrands, \cite{Almgren68,Almgren76,schoensimonalmgren,HarrisonPugh15,HarrisonPugh16}.
An existence and regularity result in arbitrary dimension and codimension for homological boundary constraints was achieved by Almgren in \cite{Almgren68} using as tool the space of varifolds, encoding the notion of multiplicity.

The aim of this paper is to extend the aforementioned results \cite{DeLGhiMag,DePDeRGhi,DeLDeRGhi16,DePDeRGhi3} to the minimization of an anisotropic energy on classes of rectifiable varifolds in any dimension and codimension, see Theorem \ref{thm generale}. The limit of a minimizing sequence of varifolds with density uniformly bounded from below is proven to be rectifiable. Moreover, with the further assumption that the minimizing sequence is made of integral varifolds with uniformly locally bounded anisotropic first variation, the limiting varifold turns out to be also integral.

We remark that every sequence of rectifiable (resp. integral) varifolds enjoying a uniform bound on the mass and on the isotropic first variation is precompact in the space of rectifiable (resp. integral) varifolds. This has been proved by Allard in \cite[Section 6.4]{Allard}, see also \cite[Theorem 42.7 and Remark 42.8]{SimonLN}. 

One of the main results of this work is indeed an anisotropic counterpart of the aforementioned compactness for integral varifolds, in the assumption that the limiting varifold has positive lower density, see Theorem \ref{integrality}. 

The additional tool available in the isotropic setting is the monotonicity formula for the mass ratio of stationary varifolds, which ensures that the density function is upper semicontinuous with respect to the convergence of varifolds. This property allows the limiting varifold to inherit the lower density bound of the sequence.

The monotonicity formula is deeply linked to the isotropic case, see \cite{Allardratio}. Nonetheless, given a minimizing sequence of varifolds for an elliptic integrand, we are able to get a density lower bound for the limiting varifold via a deformation theorem for rectifiable varifolds, see Theorem \ref{deformationcube}. We can obtain it modifying \cite[Proposition 3.1]{DavidSemmes}, proved by David and Semmes for closed sets. Thanks to the density lower bound and the anisotropic stationarity of the limiting varifold, we can conclude directly its rectifiability applying the main theorem of \cite{DePDeRGhi2}, see Theorem \ref{thm:rect}. The integrality result requires additional work, see Lemma \ref{integrality}: the idea is to blow-up every varifold of the minimizing sequence in a point in which the limiting varifold has Grassmannian part supported on a single $d$-plane $S$ (note that this property holds $\|V\|$-a.e. by the previously proved rectifiability). Applying a result proved in \cite{DePDeRGhi2}, see Lemma \ref{t:integrality}, on a diagonal sequence of blown-up varifolds, we get that roughly speaking their projections on $S$ converge in total variation to an $L^1$ function on $S$. This function is integer valued thanks to the integrality assumption on the minimizing sequence and coincides with the density of the limiting varifold in the blow-up point, which is consequently an integer. Since the argument holds true for $\|V\|$-a.e. point, the limiting varifold turns out to be integral.

\begin{ack}
The author is grateful to Guido De Philippis for useful discussions. This work has been supported by SNF 159403 {\it Regularity questions in geometric measure theory}. 
\end{ack}

\section{Notation and main result}
We will always work in \(\R^n\) and \(1\le d\le n \) will always be an integer number.
For any subset $X\subseteq \R^n$, we denote $\overline X$ its closure, Int$(X)$ its interior and $X^c:=\R^n\setminus X$ its complementary set.

We are going to use the following notation:  $Q_{x,l}$  denotes the closed cube centered in the point $x\in \R^n$, 
with edge length $l$; moreover we set 
\begin{equation}\label{rettangoli}
\quad B_{x,r}:=\{y \in \R^n \ : \ |y-x|<r\}.
\end{equation}
When cubes and balls are centered in the origin, we will simply write  $ Q_{l}$ and $B_{r}$. Cubes and balls in the subspace $\R^d\times\{0\}^{n-d}$ are denoted with $Q_{x,l}^d$ and $B_{x,r}^{d}$ respectively. 

 For a matrix  \(A\in \R^n\otimes\R^n\),  \(A^*\) denotes its transpose. Given   \(A,B\in \R^n\otimes\R^n\) we define \( A:B=\tr A^* B=\sum_{ij} A_{ij} B_{ij}\), so that \(|A|^2= A:A\).

\subsection{Measures  and rectifiable sets}
Given a locally compact metric space $Y$, we denote by $\mathcal M_+(Y)$ the set of positive Radon measures in $Y$, namely the set of measure on the $\sigma$-algebra of Borel sets of $Y$ that are locally finite and inner regular.
In particular we consider the subset of Borel probability measures $\mathcal M_P(Y)\subset \mathcal M_+(Y)$, namely $\mu \in \mathcal M_P(Y)$ if $\mu \in \mathcal M_+(Y)$ and $\mu(Y)=1$.

For a  Borel set \(E\),  \(\mu\res E\) is the  restriction of \(\mu\) to \(E\), i.e. the measure defined by \([\mu\res E](A)=\mu(E\cap A)\). 

For an \(\R^m\)-valued Radon measure on $\R^n$, \(\mu\in \mathcal M(\R^n,\R^m)\), we denote by \(|\mu|\in \mathcal M_+(\R^n)\) its total variation and we recall that, for every open subset \(U\subseteq \R^n\),
\[
|\mu|(U)=\sup\Bigg\{ \int \langle \varphi(x) ,d\mu(x)\rangle\,:\quad \varphi\in C_c^\infty(U,\R^m),\quad \|\varphi\|_\infty \le 1 \Bigg\}.
 \] 
Eventually, we denote by  $\H^d$   the  $d$-dimensional Hausdorff measure and for a \(d\)-dimensional vector space \(T\subseteq \R^n\) 
we will often identify \(\H^d\res T\) with the \(d\)-dimensional Lebesgue measure \(\mathcal L^d\) on \(T\approx \R^d\).

A  set \(K\) is said to be \(d\)-rectifiable if it can be covered, 
up to an \(\H^d\)-negligible set, by countably many \(C^1\) $d$-dimensional 
submanifolds.   Given a $d$-rectifiable set $K$, we denote $T_xK$ the approximate tangent space of $K$ at $x$, which exists for $\H^d$-almost every point $x \in K$, \cite[Chapter 3]{SimonLN}. 

For  \(\mu\in \mathcal M_+(\R^n) \) we consider its lower and upper  \(d\)-dimensional densities at \(x\):
\[
\Theta_*^d(x,\mu)=\liminf _{r\to 0} \frac{\mu(B_{x,r})}{ \omega_d r^d}, \qquad \Theta^{d*}(x,\mu)=\limsup_{r\to 0} \frac{\mu(B_{x,r})}{ \omega_d r^d},
\]
where \(\omega_d=\H^d(B^d)\) is the measure of the \(d\)-dimensional unit ball in \(\R^d\). In case these  two limits are equal, we denote by \(\Theta^d(x,\mu)\) their common value. Note that, if $\mu= \theta\H^d \res K$ with \(K\) $d$-rectifiable, then \(\theta(x)=\Theta_*^d(x,\mu)=\Theta^{d*}(x,\mu)\) for \(\mu\)-a.e. \(x\), see~\cite[Chapter 3]{SimonLN}, and $\mu$ is denoted as a $d$-rectifiable measure.

If \(\eta :\R^n\to \R^n\) is a Borel map and \(\mu\) is a Radon measure, we let \(\eta_\# \mu=\mu\circ \eta^{-1}\) be the push-forward of \(\mu\) through \(\eta\).

\subsection{Varifolds and integrands}
 We denote by $G=G(n,d)$ the Grassmannian of unoriented $d$-dimensional hyperplanes in $\R^{n}$ and, for every $U\subseteq \R^n$ we define $G(U):=U\times G$.
 
We define the space of $d$-varifold as $\mathbf V_d=\mathcal M_+(G(\R^n))$. A $d$-varifold $V\in \mathbf V_d$ is said $d$-rectifiable if there exists a $d$-rectifiable set $M$ and a function $\theta\in L^1(\R^n;\R^+;\H^d\res M)$, such that 
\begin{equation}\label{rappresenta}
V=\theta \H^d \res M \otimes \delta_{T_xM}.
\end{equation}
We denote with $\mathbf R_d\subseteq \mathbf V_d$ the subset of the $d$-rectifiable varifolds.

Moreover we say that a $d$-rectifiable varifold $V$ is integral, or equivalently $V \in \mathbf I_d$, if in the representation \eqref{rappresenta}, the density function $\theta$ is also integer valued.

Given $V \in \mathbf V_d$ and a map $\psi \in C^1(\R^n;\R^n)$, we define the push forward $\psi_\# V \in \mathbf V_d$ of $V$ with respect to $\psi$ as
$$\int_{G(\R^n)}\phi(x,T)d(\psi_\#V)(x,T)=\int_{G(\R^n)}\phi(\psi(x),d\psi_x(T))J\psi(x,T) dV(x,T),\quad \forall \phi\in C^0_c(G(\R^n)),$$
 where $d\psi_x(T)$ is the image of $T$ under the linear map $d\psi_x$ and $J\psi(x,T)$ denotes the $d$-Jacobian determinant of the differential $d\psi_x$ restricted to the $d$-plane $T$, see \cite{SimonLN}.
Note that the push-forward of a varifold \(V\) is {\em not} the same as the push-forward of the  Radon measure \(V\) through a map $\psi$ defined on the Grassmannian bundle $G(\R^n)$ (the meaning of push-forward will be always clear by the domain of the map $\psi$ and in both cases denoted with \(\psi_\# V\)).

The symbol $\|V\|$ will denote the Radon measure $\pi_\# V$ on $\R^n$, 
where $\pi : G(\R^n) \rightarrow  \R^n : (x,T) \mapsto x$ 
is the projection on the first factor.

If $V \in \mathbf R_d$, it is representable as in \eqref{rappresenta} and consequently we can extend the notion of push forward with respect to maps $\psi:\R^n \to \R^n$ which are merely Lipschitz as follows (see \cite[Section 15]{SimonLN}):
$$\psi_\#V:=\tilde \theta \H^d \res \psi(M) \otimes \delta_{T_x\psi(M)},\qquad \mbox{where $\tilde \theta(x):=\int_{\psi^{-1}(x)\cap M} \theta d\H^0$.}$$
We remark that $\tilde \theta$ is defined for $\H^d$-a.e. point of $\psi(M)$.

Observe that the following equality holds (see \cite[Section 15]{SimonLN}):
\begin{equation}\label{mass push}
\|\psi_\#V\|(A)=\int_{\psi(M)\cap A}\tilde \theta d\H^d=\int_{M\cap \psi^{-1}(A)} \theta J_M\psi d\H^d, \quad \forall \mbox{ Borel set } A\subseteq \R^n,
\end{equation}
where $J_M\psi(y)$ denotes the Jacobian determinant of the tangential differential $d_M\psi_y: T_yM\to \R^{n}$, see \cite[Sections 12 and 15]{SimonLN}.

Given $V \in \mathbf V_d$, we define respectively the lower and upper densities of $V$ at a point $x$ as follows
\[
\Theta^d_*(x,V)=\Theta^d_*(x,\|V\|)\qquad\textrm{and}\qquad\Theta^{d*}(x,V)=\Theta^{d*}(x,\|V\|).
\] 
In case $\Theta^d_*(x,V)=\Theta^{d*}(x,V)$, we denote the common value $\Theta^d(x,V)$ and it will be referred to as density of $V$ at $x$.
Note that if $V \in \mathbf R_d$, then \(\Theta^d_*(x,V)=\Theta^{d*}(x,V)=\Theta^d(x,V)\) for \(\|V\|\)-a.e. \(x\), see~\cite[Chapter 3]{SimonLN}.

We will call concentration set of $V \in \mathbf R_d$ the set 
$$\mbox{conc}(V):=\{x \in \R^n \, : \, \Theta^d_*(x,V)>0\},$$
and we will equivalently say that $V$ is concentrated on $\mbox{conc}(V)$.

The anisotropic Lagrangians that we consider are \(C^1\) integrand 
$$
F: G(\R^n) \longrightarrow R_{>0}:=(0,+\infty),
$$
for which there exist two positive constants \(\lambda, \Lambda\) such that
\begin{equation}\label{cost per area}
0 < \lambda \leq F(x,T) \leq \Lambda<\infty\qquad\textrm{for all \((x,T)\in G(\R^n)\).}
\end{equation}
Given $V \in \mathbf V_d$ and an open subset $U\subseteq \R^n$, we define:
\begin{equation}\label{energia}
\FF(V,U) := \int_{G(U)} F(x,T)\, dV(x,T) \mbox{ \ \ and \ \ } \FF(V) := \FF(V,\R^{n}).
\end{equation}
For a vector field \(g\in C_c^1(\R^n,\R^n)\), we consider the family of functions \(\varphi_t(x)=x+tg(x)\), and we note that they are diffeomorphisms of \(\R^n\) into itself. The {\em anisotropic first variation} is  defined as the following linear operator on $C_c^1(\R^n;\R^n)$:  
\[
\delta_F V(g):=\frac{d}{dt}\FF\big ( \varphi_t^{\#}V\big)\Big|_{t=0}.
\]
It can be easily shown, see \cite[Appendix A]{DePDeRGhi2}, that
 \begin{equation}\label{eq:firstvariation}
\delta_F V(g) = \int_{G(\Omega)} \Big[\langle d_xF(x,T),g(x)\rangle+ B_F(x,T):Dg(x)  \Big] dV(x,T),
 \end{equation}
where  the matrix \(B_F(x,T)\in \R^n\otimes\R^n\) is uniquely defined.

 We are going to use the  following properties of  \(B_F(x,T)\), see \cite[Section 2.2]{DePDeRGhi2}:
\begin{equation}\label{operatore}
|B_F(x,T)|\le |dF(x,T)| |T| \qquad  \textrm{for all \((x,T) \in G(\R^n)\) },
\end{equation}
and, if we define the Lagrangian \(F_r(z,T):=F(x+rz,T)\), then 
 \begin{equation}\label{eq:firstvariationfroze}
B_F(x+rz,T)=B_{F_r}(z,T)\qquad  \textrm{for all \((z,T) \in G(\R^n)\)}.
 \end{equation}
Given an open subset $U\subset \R^n$, we will say that \(V\) is {\em \(\FF\)-stationary} in $U$ if \(\delta_FV(g) =0\) for every \(g\in C_c^1(U,\R^n)\).

The anisotropic Lagrangians that we will use in the sequel are required to verify the following ellipticity property (called \emph{atomic condition}) at every point \(x\in \R^n\):
\begin{definition}\label{atomica}
Given an integrand  $F\in C^1(G(\R^n)) $, \(x\in \R^n\)  and a Borel  probability measure $\mu \in \mathcal M_P (G(n,d))$, let us define  
\begin{equation}\label{eq:A}
A_x(\mu):=\int_{G(n,d)} B_F(x,T)d\mu(T)\in \R^n\otimes \R^n.
\end{equation}
We say that \(F\) verifies the \emph{atomic condition} $(AC)$ at \(x\) if the following two conditions are satisfied:
\begin{itemize}
\item[(i)]  \(\dim\ker A_x(\mu)\le n-d\) for all \(\mu \in \mathcal M_P (G(n,d))\),
\item[(ii)]  if  \(\dim\ker A_x(\mu)= n-d\), then  \(\mu=\delta_{T_0}\) for some \(T_0\in G(n,d)\).
\end{itemize}
\end{definition}
This condition has been introduced in \cite[Definition 1.1]{DePDeRGhi2} and it has been proved to be necessary and sufficient to obtain an Allard type rectifiability result, see \cite[Theorem 1.2]{DePDeRGhi2}. 

We require the atomic condition in Definition \ref{atomica} in order to conclude the rectifiability of the limiting varifold in Theorem \ref{thm generale}.

\subsection{Competitors}
Throughout all the paper, $H\subseteq \mathbb R^{n}$ will denote a closed subset of $\R^{n}$. 
Assume to have a class of varifolds $\mathcal{P} (H,F)\subseteq \mathbf R_d$ encoding a notion of boundary: 
one can then formulate the anisotropic Plateau problem by asking whether the infimum
\begin{equation}
  \label{plateau problem generale}
m_0 :=  \inf \big\{\FF(V) : V\in \mathcal{P} (H,F)\big \}
\end{equation}
is achieved by some varifold (which is the limit of a minimizing sequence), if it belongs to the chosen class $\P(H)$ and which additional  regularity properties  it satisfies. We will say that a sequence $(V_j)_{j\in \N} \subseteq \mathcal{P} (H,F)$ is a {\em minimizing sequence} if   $\FF(V_j) \downarrow m_0$.

We need to introduce some minimal requirements for the class $\mathcal{P} (H,F)$. Roughly speaking,  $\mathcal{P} (H,F)$ has to be closed by a space of deformations that we define as in \cite{DePDeRGhi}:
\begin{definition}[Lipschitz deformations]\label{d:deform}
Given a ball $B_{x,r}$, we let $\mathfrak D(x,r)$ be the set of functions $\vphi:\R^n \rightarrow \R^n$ 
such that $\varphi(z)=z$ in $\R^n\setminus B_{x,r}$ and which are smoothly isotopic 
to the identity inside $B_{x,r}$, namely those for which there exists an isotopy $\lambda \in C^\infty([0,1]\times \R^n;\R^n)$ such that 
$$\lambda(0,\cdot) = \Id, \quad \lambda(1,\cdot)=\vphi, \quad \lambda(t,h)=h \quad\forall\,(t,h)\in [0,1]\times (\R^n \setminus B_{x,r}) \quad \mbox{ and } $$
$$ \lambda(t,\cdot) \mbox{\ is a diffeomorphism of } \R^n \ \forall t \in [0,1].
$$
We finally set $\D(x,r):=\overline{\mathfrak D(x,r)}^{C^0}\cap \Lip (\mathbb R^{n})$, the sequential closure of $\mathfrak D(x,r)$ with respect to the uniform convergence, intersected with the space of Lipschitz maps.
\end{definition}

The classes that we are going to consider can be now defined:

\begin{definition}[Deformed competitors and good class]\label{def good class}
Let $H\subseteq \mathbb R^{n}$ be a closed set and  $V\in \mathbf R_d$.
A  {\em deformed competitor} for $V$ in $B_{x,r}$ is any varifold
\begin{equation*}
 \varphi_\#  V \in \mathbf R_d \quad \mbox{ where } \quad \varphi \in \D(x,r).
\end{equation*}
We say that $\mathcal{P} (H,F)$ is a {\emph good class} with respect to $H$ and $F$ if $\mathcal{P} (H,F)\subseteq \mathbf R_d$ and for every $V\in\mathcal{P} (H,F)$ it holds:
\begin{itemize}
\item $\mbox{conc}(V)$ is a relatively closed subset of $\R^n\setminus H$;
\item for every $x\in \R^n \setminus H$ and for a.e. $r\in (0, \dist (x, H))$
\begin{equation}
  \label{inf good class}
  \inf \big\{ \FF(W) : W\in \mathcal{P} (H,F)\,,W \res G(\overline{B_{x,r}}^c) =V\res G(\overline
{B_{x,r}}^c) \big\} \leq \FF(L), \,
\end{equation}
whenever $L$ is any deformed competitor for $V$ in $B_{x,r}$.
\end{itemize}  
\end{definition}

\begin{remark}\label{r:lagr}
Given $V\in \mathbf R_d$ and a deformation $\varphi \in \D(x,r)$, using property \eqref{cost per area}, we deduce the quasiminimality property
 \begin{equation}\label{quasiminimo}
 \FF(\varphi_\#V)\leq \Lambda \|\varphi_\#V\|(\R^n) \leq \Lambda (\Lip(\varphi))^d \|V\|(\R^n)  \leq \frac{\Lambda}{\lambda} (\Lip(\varphi))^d \FF(V).
 \end{equation}
\end{remark}

\subsection{The main result}
We can now state our main result:
\begin{theorem}\label{thm generale}
Let $F\in C^1(G(\R^n))$ be a Lagrangian satisfying the atomic condition at every point \(x\in \R^n\) and enjoying the bounds \eqref{cost per area}. Let $H\subseteq \mathbb R^{n}$ be a closed set and $\mathcal{P} (H,F)$ be a good class with respect to $H$ and $F$. Assume the infimum in Plateau problem \eqref{plateau problem generale} is finite and let 
$(V_j)_{j\in \N}\subseteq \mathcal{P} (H,F)$ be a minimizing sequence. Then, up to subsequences, $V_j$ converges to a $d$-varifold $V \in \mathbf V_d$ with the following properties:
\begin{itemize}
\item[(a)] $\liminf_j \FF(V_j) \geq  \FF(V)$;
\item[(b)] if $V \in \mathcal{P} (H,F)$, then $V$ is a minimum for \eqref{plateau problem generale};
\item[(c)] $V$ is \(\FF\)-stationary in $\R^{n}\setminus H$.
\end{itemize}
Furthermore:
\begin{itemize}
\item[(d)] if the minimizing sequence $(V_j)_{j\in \N}$ enjoys a uniform density lower bound in $\R^n\setminus H$, i.e. there exists $\delta>0$ such that:
$$\Theta^d(x,V_j)\geq \delta, \qquad \mbox{for } \|V_j\|\mbox{-a.e. } x\in \R^n\setminus H, \, \forall j\in \N,$$
then $V\res G(\R^n\setminus H)\in \mathbf R_d$ and $\mbox{conc}(V)$ is relatively closed in $\R^n\setminus H$;
\item[(e)] if the minimizing sequence $(V_j)_{j\in \N}$ satisfies $(V_j\res G(\R^n\setminus H))_{j\in \N} \subseteq \mathbf I_d$ and  
\begin{equation}\label{assumption e}
\sup_j |\delta_FV_j|(W) <\infty, \qquad \forall \, W \subset \subset \R^n \setminus H,
\end{equation}
then $V\res G(\R^n\setminus H)\in \mathbf I_d$.
\end{itemize}
\end{theorem}

\begin{remark}
If the assumption $(V_j\res G(\R^n\setminus H))_{j\in \N} \subseteq \mathbf I_d$ required in the condition $(e)$ of Theorem \ref{thm generale} is satisfied, also condition $(d)$ applies, with the trivial density lower bound $\delta=1$.
\end{remark}

\textbf{Open Question}
One may conjecture that, given an integral minimizing sequence in a good class $\mathcal P(H)$, there always exists a minimizing sequence in $\mathcal P(H)$ satisfying \eqref{assumption e}.

 The idea is to perform a kind of motion by mean curvature and is described in the following example (for simplicity $F$ is assumed to be the area functional): fix a ball $B$ and assume that the limiting varifold $V$ restricted to $B$ is the one density varifold naturally associated to one of the diameters of $B$ and that the general element $V_j$ of the minimizing sequence is the one density varifold associated to the union of the same diameter and of $j$ disjoint circles of radius $\frac{1}{j^2}$ contained in the ball. Of course
 $$\|V_j\|(B) = \|V\|(B)+2\pi \frac{j}{j^2} \to  \|V\|(B),$$
 but
  $$|\delta_FV_j|(B)=2\pi j$$
 and \eqref{assumption e} does not hold.
 Nevertheless, one can deform every $V_j$, collapsing each of the $j$ circles in its origin, via a Lipschitz map. This procedure generates a new minimizing sequence, which is still in the class, and that satisfies \eqref{assumption e}.
 
 Of course, for a general minimizing sequence this argument is far from being a proof and the conjecture seems to the author quite a delicate task.

\section{Preliminary results}\label{notation}

A key result we are going to use is a deformation theorem for rectifiable varifolds with density bigger or equal than one, that we prove in this section. It is the analogous of the deformation theorem for closed sets, due to David and Semmes \cite[Proposition 3.1]{DavidSemmes}, and of the one for rectifiable currents \cite{SimonLN,FedererBOOK}. 

The proof relies on the one of \cite[Proposition 3.1]{DavidSemmes}.

Before stating the theorem, let us introduce some further notation.
Given a closed cube $Q=Q_{x,l}$  and $\e > 0$, we cover $Q$ with a grid of closed smaller cubes 
with edge length $\e \ll l$, with non empty intersection with $\mbox {Int} (Q)$ 
and such that the decomposition is centered in $x$ (i.e. one of the 
subcubes is centered in $x$). The family of this smaller cubes is denoted $\L_\e(Q)$. 
We set 
\begin{equation}\label{cornici}
\begin{gathered}
C_1:=\bigcup \left \{T \cap Q: T \in \L_\e(Q), T \cap \pa Q \neq \emptyset  \right \},\\
C_2:=  \bigcup \left \{T \in \L_\e(Q) :  (T\cap Q) \not \subseteq C_1, T \cap \pa C_1  \neq \emptyset \right \},\\
Q^1:=\overline{Q \setminus (C_1 \cup C_2)}
\end{gathered}
\end{equation}
and consequently
$$\L_\e(Q^1 \cup C_2):= \left \{T \in \L_\e(Q) : T \subseteq (Q^1 \cup C_2)\right \}.$$
For each nonnegative integer $m\leq n$, let $\L_{\e,m}(Q^1 \cup C_2)$ denote the collection 
of all $m$-dimensional faces of cubes in $\L_\e(Q^1 \cup C_2)$ and 
$\L^*_{\e,m}(Q^1 \cup C_2)$ will be the set of the elements of $\L_{\e,m}(Q^1 \cup C_2)$ 
which are not contained in $\pa (Q^1 \cup C_2)$. We also let  $S_{\e,m}(Q^1 \cup C_2):= \bigcup \L_{\e,m}(Q^1 \cup C_2)$ be the $m$-skeleton of order $\e$ in $Q^1 \cup C_2$.

\begin{theorem}\label{deformationcube}
Given $x_0\in \R^n$, $r>0$, a closed cube $Q \subseteq B_{x_0,r}$ and $V\in \mathbf R_d$ such that: 
$$V:=\theta \H^d\res K\otimes \delta_{T_xK}, \qquad \mbox{where}\qquad  \theta(x) \geq 1 \quad \mbox{for }\H^d\res K-\mbox{a.e. } x\in Q,$$
$$\mbox{ $K\cap Q$ is a closed set \quad and \quad $\|V\|(Q)<+\infty$}.$$ 
 
Then there exists a map $\Phi_{\e,V} \in \D(x_0,r)$ satisfying the following properties:
\begin{itemize}
\item[(1)]  $\Phi_{\e,V}(x)=x \ \mbox{for} \ x \in \R^n \setminus (Q^1 \cup C_2)$;
\item[(2)]  $\Phi_{\e,V}(x)=x \ \mbox{for} \ x \in S_{\e,d-1}(Q^1 \cup C_2)$;
\item[(3)]  $\Phi_{\e,V}(K\cap (Q^1 \cup C_2))\subseteq S_{\e,d}(Q^1 \cup C_2)\cup \pa(Q^1 \cup C_2)$;
\item[(4)]  $\Phi_{\e,V}(T)\subseteq T$ for every $T \in \L_{\e,m}(Q^1 \cup C_2)$, with $m=d,...,n$;
\item[(5)] $\|(\Phi_{\e,V})_\#V\|(T)\leq k_1\|V\|(T)$ for every $T \in \L_{\e}(Q^1 \cup C_2)$;
\item[(6)]  either $\|(\Phi_{\e,V})_\#V\|(T)=0$ or $\|(\Phi_{\e,V})_\#V\|(T)\geq \H^d(T)$, for every $T \in \L^*_{\e,d}(Q^1)$;
\end{itemize}
where $k_1$ depends only on $n$ and $d$ (but neither on $\e$ nor on $V$).
\end{theorem}

\begin{proof}
Our map $\Phi_{\e,V}$ can be obtained as the last element of a finite sequence $\Phi_n,\Phi_{n-1},...,$ $\Phi_d,\Phi_{d-1}$ of Lipschitz maps on $\R^n$.
The maps $\Phi_m$ with $m=d,...,n$ will satisfy the analogous of $(1)-(5)$, with $(2)$ and $(3)$ replaced by
$$
\Phi_m(x)=x \ \mbox{for} \ x \in S_{\e,m}(Q^1 \cup C_2),
$$
$$
\Phi_m(K\cap (Q^1 \cup C_2))\subseteq S_{\e,m}(Q^1 \cup C_2)\cup \pa(Q^1 \cup C_2).
$$
The last map $\Phi_{d-1}$ will be constructed in order to satisfy also property $(6)$.

We start with $\Phi_n(x):=x$, which verifies all the required conditions. Suppose that, for a given $m>d$, we have already built $\Phi_n,\Phi_{n-1},...,\Phi_m$. We want to define $\Phi_{m-1}$ as
\begin{equation}\label{mappa}
\Phi_{m-1}:= \psi_{m-1}\circ \Phi_{m},
\end{equation}
where $\psi_{m-1}$ is a Lipschitz map in $\R^n$ given by the following Lemma:
\begin{lemma}\label{intermedio}
The exists a Lipschitz map $\psi_{m-1}:Q^1 \cup C_2 \to Q^1 \cup C_2$ such that:
$$
\psi_{m-1}(x)=x \ \mbox{for} \ x \in S_{\e,m-1}(Q^1 \cup C_2)\cup \pa(Q^1 \cup C_2),
$$
$$
\psi_{m-1}(\Phi_m(K\cap (Q^1 \cup C_2)))\subseteq S_{\e,m-1}(Q^1 \cup C_2)\cup \pa(Q^1 \cup C_2),
$$
$$
\psi_{m-1}(T)\subseteq T \mbox{  for every }T \in \L_{\e,m}(Q^1 \cup C_2), \mbox{ with }m=d,...,n,
$$
and
\begin{equation}\label{difference}
\|(\psi_{m-1}\circ \Phi_{m})_\#V\|(T)\leq C\|(\Phi_m)_\#V\|(T) \mbox{ for every }T \in \L_{\e}(Q^1 \cup C_2),
\end{equation}
where $C$ depends only on $m$ and $d$.
\end{lemma}

Assuming Lemma \ref{intermedio}, we can easily extend $\psi_{m-1}$ to be the identity outside $Q^1 \cup C_2$ and the map $\Phi_{m-1}$ defined in \eqref{mappa} satisfies the desired properties. 

To conclude, we need to construct $\Phi_{d-1}$ in order to satisfy also condition $(6)$. We proceed in a way analogous to the one used in \cite[Theorem 2.4]{DePDeRGhi}.

We want to set $$\Phi_{d-1}:=\Psi \circ \Phi_{d},$$
where $\Psi$ will be defined below.
We first define $\Psi$ on every $T \in \L_{\e,d}(Q^1 \cup C_2)$ distinguishing two cases
\begin{itemize}
\item[(a)]  if either $\|(\Phi_{d})_\#V\|(T)=0$ or $\|(\Phi_{d})_\#V\|(T)\geq \H^d(T)$ or $T \not \in \L^*_{\e,d}(Q^1)$, then we set $\Psi_{|T}=\Id$;
\item[(b)]  otherwise, since the varifold density $\Theta (x,V)$ is bigger or equal than one for $\|V\|$-a.e. $x \in Q$, the same holds for $(\Phi_{d})_\#V$, because $\Phi_{d}$ is a Lipschitz map. We infer that 
$$\H^d(T)> \|(\Phi_{d})_\#V\|(T)\geq \H^d(\Phi_{d}(K\cap Q)\cap T).$$
 Since $\Phi_{d}(K\cap Q)$ is compact ($K\cap Q$ is compact by assumption), there exists $y_T \in T$ and $\delta_T>0$ such that $B_{\delta_T}(y_T)\cap \Phi_{d}(K\cap Q)=\emptyset$; we define 
$$\Psi_{|T}(x)=x+\alpha(x-y_T) \min \left \{1,\frac{|x-y_T|}{\delta_T}\right\},$$
where $\alpha>0$ such that the point $x+\alpha(x-y_T)\in  \left (\pa T\right ) \times \{0\}^{n-d}$.
\end{itemize}
The second step is to define $\Psi$ on every $T' \in \L_{\e,d+1}(Q^1 \cup C_2)$. Without loss of generality, we can assume $T'$ centered in $0$. 
We divide $T'$ in pyramids $P_{T,T'}$ with base $T\in \L_{\e,d}(Q^1 \cup C_2)$ and vertex $0$. Assuming $T \subseteq \{x_{d+1}=-\frac \e 2, x_{d+2},...,x_n=0\}$ and $T' \subseteq \{x_{d+2},...,x_n=0\}$, we set
$$\Psi_{|P_{T,T'}}(x)=-\frac{2x_{d+1}}{\e}\Psi_{|T}\left(-\frac{x}{x_{d+1}}\frac \e 2\right).$$
We iterate  this procedure on all the dimensions till to $n$, defining it well in $Q^1 \cup C_2$. 
Since $\Psi_{|\pa (Q^1 \cup C_2)}=\Id$, we can extend the map as the identity outside $Q^1 \cup C_2$.

By construction of $\Psi$, if we denote 
$$(\Psi \circ \Phi_{d})_\#V=\tilde \theta \H^d \res (\Psi \circ \Phi_{d})(K) \otimes \delta_{T_x(\Psi \circ \Phi_{d})(K)},$$
 we get that $\tilde \theta = 0$ in the interior of $T$, and we can assume this is true also at the boundary since $\H^d(\partial T)=0$ and $\tilde \theta$ is defined $\H^d$-a.e..

We consequently get:
$$\|(\Psi \circ \Phi_{d})_\#V\|(T)=\int_{(\Psi \circ \Phi_{d})(K)\cap T}\tilde \theta d\H^d=0,$$
and so property $(6)$ is now satisfied.

In addition, one can easily check that $\Psi \in \D(x_0,r)$ and thus, since  \(\Phi_{d}\in \D(x_0,r)\) and the class \( \D(x_0,r)\)  is closed by composition, then also $\Phi_{d-1}\in \D(x_0,r)$. 

This concludes the proof of Theorem \ref{deformationcube} provided we prove Lemma \ref{intermedio}.

The proof of Lemma \ref{intermedio} can be repeated verbatim as the proof of \cite[Lemma 3.10]{DavidSemmes}, if we replace \cite[Lemma 3.22]{DavidSemmes} with the following:
\begin{lemma}\label{finale}
Let $T$ be an $m$-dimensional closed cube with $m>d$ and define $F:=K\cap T$.

For every $z \in T\setminus F$, we define $\e_z:=d(z,F)>0$. We consider a map $\eta_{z,T}:T \to T$ satisfying the conditions:
$$\eta_{z,T}(x) \in \partial T, \qquad \eta_{z,T}(x)-x=c(x-z), \qquad c=c(x,z,T)>0, \quad \forall x \in T\setminus B_{z,\e_z}.$$
In $B_{z,\e_z}$ we define $\eta_{z,T}$ in order to get a Lipschitz map on $T$.

Then
\begin{equation}\label{stima}
\int_{z\in \left (\frac 12 T\right )\setminus F} \|(\eta_{z,T})_\#V\|(T)\, d\H^m(z) \leq C(diam (T))^m\|V\|(T),
\end{equation}
where $C$ depends just on $m$ and $d$.
\end{lemma}

\begin{proof}[Proof of Lemma \ref{finale}]
For a given point $z$, if we denote 
$$(\eta_{z,T})_\#V=\tilde \theta \H^d \res \eta_{z,T}(K) \otimes \delta_{T_x\eta_{z,T}(K)},$$ 
by \eqref{mass push} we compute
\begin{equation}\label{seconda}
\|(\eta_{z,T})_\#V\|(T)= \int_{K\cap T} \theta J_K\eta_{z,T} d\H^d.
\end{equation}
Moreover, for every $x \in T\setminus \overline{B_{z,\e_z}}$, we have
\begin{equation}\label{prima}
\begin{split}
J_K\eta_{z,T}(x,\pi)&\leq C|D\eta_{z,T}|^d \leq C\left (\lim_{y\to x}\frac{|\eta_{z,T}(x)-\eta_{z,T}(y)|}{|x-y|}\right )^d\\
&\leq C\left (\lim_{y\to x}\frac{|x-y|diam (T)}{|x-y|\cdot |x-z|}\right )^d\leq C\frac{(diam (T))^d}{|x-z|^d},\end{split}
\end{equation}
where $C$ depends just on $m$ and $d$.
Plugging \eqref{prima} in \eqref{seconda}, we infer that 
$$
\|(\eta_{z,T})_\#V\|(T)\leq  C(diam (T))^d\left \|\frac{1}{|\cdot-z|^d}V\right \|(T).
$$ 
Integrating this estimate over $\left (\frac 12 T\right )\setminus F$ and applying Fubini's theorem, we get
$$
\int_{z\in \left (\frac 12 T\right )\setminus F} \|(\eta_{z,T})_\#V\|(T)\, d\H^m(z) \leq C(diam (T))^d \int_{T} \left (\int_T \frac{1}{|x-z|^d} \, d\H^m(z)\right )\, d\|V\|(x).
$$
Since the integral in $z$ on the right hand side is finite because $m>d$ and its value is less or equal than $C(diam (T))^{m-d}$, we conclude the estimate \eqref{stima} as we wanted to prove. 
\end{proof}
Lemma \ref{finale} allows us to prove Lemma \ref{intermedio} as for \cite[Lemma 3.10]{DavidSemmes}. Our proof is now concluded.
\end{proof}

For further purposes, we recall two important results proven in \cite{DePDeRGhi2}. The first one was inspired by the ``Strong Constancy Lemma'' of Allard~\cite[Theorem 4]{Allard84BOOK}:

\begin{lemma}\label{t:integrality}\emph{(\cite[Lemma 3.2]{DePDeRGhi2})}
Let \(F_j:G(B_1)\to \R_{>0}\) be a sequence of \(C^1\) integrands and let   
$V_j\in \mathbf V_d(G(B_1))$ be a sequence of $d$-varifolds equi-compactly supported in $B_1$
 (i.e. such that \(\spt \|V_j\|\subset K \subset\subset B\)) with $\|V_j\|(B_1)\leq 1$.  
If there exist $N>0$ and  $S\in G(n,d)$ such that
\begin{itemize}
\item [(1)] $|\delta_{F_{j}} V_j |(B_1) +\|F_j\|_{C^1(G(B_1))}\leq N$,
\item [(2)] \(|B_{F_j}(x,T)-B_{F_j}(x,S)|\le \omega(|S-T|)\) for some modulus of continuity independent on \(j\),
\item [(3)] $\delta_j:=\int_{G(B_1)} |T-S|  dV_j(z,T) \rightarrow 0$ as \(j\to \infty\),
\end{itemize}
then, up to subsequences, there exists $\gamma \in L^1(B^d_1, \H^d\res B^d_1)$ such that for every $0<t<1$ 
\begin{equation}\label{eq:limiteforte}
\Big| (\Pi_S)_{\#}\big(F_j(z, S)\|V_j\|\big) - \gamma\mathcal \H^d\res B^d_1\Big|(B^d_t ) \longrightarrow 0, 
\end{equation}
where \(\Pi_S:\R^n\to S\) denotes   the orthogonal projection onto \(S\).
\end{lemma}

The second result is the anisotropic counterpart of Allard's rectifiability theorem. We state a weaker version of this result which is sufficient for our purposes, but one may observe that it has a stronger formulation, see \cite[Theorem 1.2]{DePDeRGhi2}.
\begin{theorem}\label{thm:rect}
 Let $F\in C^1( G(\R^n),\R_{>0}) $ be a positive integrand satisfying the \emph{atomic condition} as in Definition \ref{atomica} at every \(x\in \R^n\).
Given \(V\in  \mathbf V_d\) and an open set $U\subset \R^n$ such that \(\delta_FV\res U\) is a Radon measure and $\Theta^d_*(x,V)>0$ for $\|V\|$-a.e. $x \in U$, then $V\res G(U)$ is a \(d\)-rectifiable varifold.
  \end{theorem}

\section{An integrality theorem}
In this section, we prove an integrality theorem of independent interest, which is going to be applied in the proof of Theorem \ref{thm generale}.

\begin{theorem}\label{integrality}
Let $F\in C^1( G(\R^n),\R_{>0}) $ be a positive integrand satisfying the \emph{atomic condition} as in Definition \ref{atomica} at every \(x\in \R^n\). Given an open set $U\subseteq \R^n$ and a sequence of integral varifolds $(V_j)_{j\in \N}\subseteq \mathbf I_d$ converging to a varifold $V$. Assume that $V$ enjoys the density lower bound
\begin{equation}\label{densita}
\Theta^d_* (x, V ) > 0 \quad \mbox{ for $\|V\|$-a.e. $x \in U$}
\end{equation}
and that the sequence $(V_j)_{j \in \N}$ satisfies
\begin{equation}\label{variazione limitata}
\sup_{j \in \N} |\delta_FV_j|(W) <\infty, \qquad \forall \, W \subset \subset U;
\end{equation}
then $V\res G(U)\in \mathbf I_d$.
\end{theorem}

\begin{proof}

By the assumption \eqref{variazione limitata} and by the lower semicontinuity of the total variation of the anisotropic first variation with respect to varifolds convergence, we get that $\delta_FV\res U$ is a Radon measure. Moreover, $V$ enjoys the density lower bound \eqref{densita}. Since $F$ satisfies the atomic condition as in Definition \ref{atomica} at every \(x\in \R^n\), we are in the hypotheses to apply Theorem \ref{thm:rect} and to conclude that $V$ is a $d$-rectifiable varifold.

We now prove that the limiting varifold $V$ is integral.

Since $V \res G(U)\in \mathbf R_d$, it can be represented as
$$V\res G(U):=\Theta (\cdot,V)\H^d\res K\otimes \delta_{T_xK},$$
 where $K$ is a $d$-rectifiable set, $\Theta (\cdot,V) \in L^1(\R^n;\H^d)$ and $T_xK$ denotes the tangent space of $K$ at $x$.

By assumption \eqref{variazione limitata}, we know that there exists $\nu \in \mathcal M_+(U)$ such that $|\delta_FV_j|$ converges weakly in the sense of measures to $\nu$ in $U$.  By Besicovitch differentiation theorem (see \cite[Theorem 2.22]{AFP}) we get that for $\|V\|$- a.e. point $x$ in $U$
\begin{equation}\label{bes}
\limsup_{r\to 0}\frac{\nu (B_{x,r})}{\|V\|(B_{x,r})}=C_{\bar x} < + \infty.
\end{equation}

We fix a point $\bar x \in U$ such that $\Theta (\bar x,V)$ and $T_{\bar x}K$ exist, $\Theta (\bar x,V)\in (0,+\infty)$ (this is true at $\|V\|$- a.e. point in $U$ by the rectifiability of $V \res G(U)$) and such that \eqref{bes} holds. Assume w.l.o.g. that $T_{\bar x}K=\R^d \times \{0\}^{n-d}$; we denote $S:=T_{\bar x}K$ and with \(\Pi_S:\R^n\to S\) and \(\Pi_{S^\perp}:\R^n\to S^\perp\) the orthogonal projections respectively onto \(S\) and $S^\perp$.

There exists a sequence of radii $(r_k)_{k \in \N}\downarrow 0$ such that $\nu(\partial B_{\bar x,r_k})=0$ and consequently there exists $j_k:=j(r_k)$ big enough so that
\begin{equation}\label{ow1}
|\delta_FV_{j_k}|(B_{\bar x,r_k})=(1+o_{r_k}(1))\nu(B_{\bar x,r_k}).
\end{equation}
Combining \eqref{bes} and \eqref{ow1}, we obtain
$$
\limsup_{k\to \infty}\frac{|\delta_F V_{j_k}|(B_{\bar x,r_k})}{\|V\|(B_{\bar x,r_k})}=\limsup_{k\to \infty}\frac{\nu(B_{\bar x,r_k})}{\|V\|(B_{\bar x,r_k})}=C_{\bar x} < + \infty,
$$
and for $k$ big enough we conclude
\begin{equation}\label{ow2}
|\delta_F V_{j_k}|(B_{\bar x,r_k})\leq 2 C_{\bar x}\|V\|(B_{\bar x,r_k}).
\end{equation}
For every $k \in \N$, we consider the rescaling transformation $\eta^{\bar x,r_k}:\R^n\rightarrow \R^n$, $\eta^{\bar x,r_k}(y)= \frac{y-\bar x}{r_k}$. We define
$$ 
V^{k}:=\left (\eta^{\bar x,r_k}_\#V \right ) \qquad \mbox{and} \qquad V_j^{k}:=\left (\eta^{\bar x,r_k}_\#V_j \right ).
$$ 
Since $V_j\rightharpoonup V$, for every $k \in \N$ 
\begin{equation}\label{conv}
\mbox{$V_j^{k}\rightharpoonup V^{k}$ \qquad  as \qquad $j \to \infty$.}
\end{equation} 
But, since $\Theta (\bar x,V)<+\infty$, we get that $V^{k}$ are locally bounded uniformly with respect to $k$ and we infer 
$$
V^{k}\rightharpoonup \Theta (\bar x,V)\H^d\res S \otimes \delta_{S}, \qquad \mbox{as } k\to \infty.
$$
Via a diagonal argument, up to extract another (not relabeled) subsequence $j_k$, if we define $\tilde V^k:=V_{j_k}^{k}$, we get
\begin{equation}\label{eq:ok}
\|V_{j_k}\|(B_{\bar x,r_k})\leq 2\|V\|(B_{\bar x,r_k})\leq 4\Theta (\bar x,V)r_k^d,  \qquad \|\tilde V^k\|(B_1^d\times B_1^{n-d} \setminus B_1^d\times B_{\frac 12}^{n-d})=o_{r_k}(1),
\end{equation}
\begin{equation}\label{eq:ok22}
\|\tilde V^k\|(B_1^d\times B_1^{n-d} )\leq 2\Theta (\bar x,V),
\end{equation}
and the convergence
\begin{equation}\label{eq:ok1}
\tilde V^k\rightharpoonup \Theta (\bar x,V)\H^d\res S \otimes \delta_{S}, \qquad \mbox{as } k\to \infty.
\end{equation}
 We consider $\chi_1 \in C_c^\infty(B_{\sqrt 2/2}^d)$ with $\chi_1\equiv 1$ in $B_{\frac 12}^d$, $\chi_2 \in C_c^\infty(B_{\sqrt 2/2}^{n-d})$ with $\chi_2\equiv 1$ in $B_{1/2}^{n-d}$ and we define $\chi\in C_c^\infty(B_1)$ as $\chi(x):=\chi_1(\Pi_S (x))\chi_2(\Pi_{S^\perp}(x))$.
 
 We denote \(F_k(z,T)=F(\bar x+r_k z,T)\) and define the family of varifolds $W_k:= \chi \tilde V^k$ equicompactly supported in $B_1$. We claim that
\begin{equation}\label{supbound}
\sup_{k\in \N} | \delta_{F_k} W_k|(B_1)<+\infty.
\end{equation}
 Indeed, we define $\chi_k:=\chi \circ \eta^{\bar x,r_k}\in C^\infty_c(B_{\bar x,r_k})$ and for every $\varphi \in C^\infty_c(B_1,\R^n)$ we consider the map $\varphi_k:=\varphi \circ \eta^{\bar x,r_k}\in C^\infty_c(B_{\bar x,r_k},\R^n)$, so that 
 \begin{equation}\label{stupidcomp}
\|\chi_k\|_\infty\leq \|\chi\|_\infty\leq 1, \qquad r_k\|\nabla \chi_k\|_\infty \leq \|\nabla \chi\|_\infty \qquad \mbox{and} \qquad \|\varphi_k\|_\infty\leq \|\varphi\|_\infty.
\end{equation}
  Thanks to equations \eqref{operatore}, \eqref{eq:firstvariationfroze}, \eqref{ow2}, \eqref{eq:ok} and \eqref{stupidcomp}, we can compute
\begin{equation*}
\begin{split}
|\delta_{F_k}W_k(\varphi)| &\, \, \quad=| \delta_{F_k}(\chi \tilde V^k)(\varphi)|\\
&\, \, \quad=\Big| \int  \big\langle d_z F_k(z, T), \chi(z) \varphi(z)\rangle d \tilde V^k(z,T)\\
&\, \, \quad\quad+ \int B_{F_k} (z,T):D\varphi(z)\, \chi(z) d\tilde V^k(z,T)\Big|	\\
&\, \, \quad=\Big| \int  \big\langle d_z F_r(\eta^{\bar x,r_k}(y), T), \chi(\eta^{\bar x,r_k}(y)) \varphi(\eta^{\bar x,r_k}(y))\rangle J\eta^{\bar x,r_k}(y,T) d V_{j_k}(y,T)\\
&\, \, \quad \quad+ \int B_{F_k} (\eta^{\bar x,r_k}(y),T):D\varphi(\eta^{\bar x,r_k}(y))\, \chi(\eta^{\bar x,r_k}(y))J\eta^{\bar x,r_k}(y,T) d V_{j_k}(y,T)\Big|	\\
&\quad \overset{\eqref{eq:firstvariationfroze}}{=}\Big| r_k^{1-d}\int  \big\langle d_y F(y, T), \chi_k(y) \varphi_k(y)\rangle  d V_{j_k}(y,T)\\
&\, \, \quad \quad+ r_k^{1-d}\int B_{F} (y,T):D\varphi_k(y)\, \chi_k(y) d V_{j_k}(y,T)\Big|	\\
&\, \quad =\Big| r_k^{1-d}\int  \big\langle d_y F(y, T), \chi_k(y) \varphi_k(y)\rangle d V_{j_k}(y,T)\\
&\, \, \quad \quad+ r_k^{1-d}\int B_{F} (y,T):D(\varphi_k\chi_k)(y)\, d V_{j_k}(y,T)	\\
&\quad \, \quad- r_k^{1-d}\int B_{F} (y,T):\nabla \chi_k(y) \otimes \varphi_k(y)\, d V_{j_k}(y,T)\Big|\\
&\quad \overset{\eqref{operatore}}{\leq} r_k^{1-d}|\delta_{F} V_{j_k}(\chi_k\varphi_k)|+ r_k^{1-d}\|F\|_{C^1(B_{\bar x,r_k})}\|V_{j_k}\|(B_{\bar x,r_k})\|\nabla \chi_k\|_\infty  \|\varphi_k\|_\infty     \\
&\overset{\eqref{eq:ok},\eqref{stupidcomp}}{\leq}    r_k^{1-d}|\delta_{F} V_{j_k}|(B_{\bar x,r_k})\|\varphi\|_\infty+4\|F\|_{C^1(B_{\bar x,r_k})} \Theta (\bar x,V)\|\nabla \chi\|_\infty  \|\varphi\|_\infty\\
&\quad \overset{\eqref{ow2}}{\leq}   2 r_k^{1-d}C_{\bar x}\|V\|(B_{\bar x,r_k})\|\varphi\|_\infty+4\|F\|_{C^1(B_{\bar x,r_k})} \Theta (\bar x,V)\|\nabla \chi\|_\infty  \|\varphi\|_\infty\\
&\quad \overset{\eqref{eq:ok}}{\leq} \big[4r_kC_{\bar x}\Theta (\bar x,V)+4\|F\|_{C^1(B_{\bar x,r_k})} \Theta (\bar x,V)\|\nabla \chi\|_\infty \big] \|\varphi\|_\infty,  
\end{split}
\end{equation*}
which implies~\eqref{supbound}. Finally, by~\eqref{eq:ok1},
\begin{equation*}\label{eq:conv}
\begin{split}
\lim_{k\to \infty}\int_{G(B_1)} |T-S| dW_k(z,T)&=\lim_{k\to \infty} \int_{G(B_1)}|T-S| \chi(z)  d\tilde V^k(z,T)\\
 &= \int_{G(B_1)}|T-S|\Theta (\bar x,V) \chi(z) d\delta_{S}(T)  d\H^d\res S(z)=0.
 \end{split}
\end{equation*}
Hence the sequence  \((W_k)_{k\in \N}\) satisfies the assumptions of Lemma~\ref{t:integrality}, indeed we observe that \(B_{F_k}(z,T)=B_{F}(\bar x+r_kz,T)\), so that assumption (2) in Lemma~\ref{t:integrality} is satisfied.
Thus we deduce that there exists $\gamma \in L^1(\H^d\res B^d_1)$ such that, along a (not relabeled) subsequence, for every \(0<t<1\)
\begin{equation}\label{eq:contradiction}
\Big| (\Pi_S)_{\#}(F(\bar x +r_k(\cdot),S)\|W_k\|)- \gamma \mathcal H^d\res B_1^d\Big| (B^d_t )\longrightarrow 0.
\end{equation}
Since $\chi_1\equiv 1$ in $B_{1/2}^d$, thanks to \eqref{eq:ok} we get
\begin{equation}
\begin{split}
(\Pi_S)_{\#}(F(\bar x +r_k(\cdot),S)&\|W_k\|)(B_{1/2}^d )\\
&=(\Pi_S)_{\#}(F(\bar x +r_k(\cdot),S)\|(\chi_2\circ \Pi_{S^\perp}) \tilde V^k\|)(B_{1/2}^d )\\
&= (\Pi_S)_{\#}(F(\bar x +r_k(\cdot),S)\|\tilde V^k\res (B_1^d\times B_1^{n-d})\|)(B_{1/2}^d ) - o_{r_k}(1),
\end{split}
\end{equation}
which we plug in \eqref{eq:contradiction} to obtain
\begin{equation}\label{eq:limiteforte1}
\Big| (\Pi_S)_{\#}(F(\bar x +r_k(\cdot),S)\|\tilde V^k\res B_1^d\times B_1^{n-d}\|)- \gamma \mathcal H^d\res B^d_1\Big| (B_{1/2}^d )\longrightarrow 0. 
\end{equation}
But, thanks to \eqref{eq:ok22} 
\begin{equation}\label{eq:limiteforte2}
\begin{split}
\Big| (\Pi_S)_{\#}&(F(\bar x +r_k(\cdot),S)\|\tilde V^k\res B_1^d\times B_1^{n-d}\|)- (\Pi_S)_{\#}(F(\bar x ,S)\|\tilde V^k\res B_1^d\times B_1^{n-d}\|)\Big| (B_{1/2}^d )\\
&\, \, =\Big| (\Pi_S)_{\#}(F(\bar x +r_k(\cdot),S)\|\tilde V^k\res B_1^d\times B_1^{n-d}\|- F(\bar x ,S)\|\tilde V^k\res B_1^d\times B_1^{n-d}\|)\Big| (B_{1/2}^d )\\
&\, \,\leq \Big| (F(\bar x +r_k(\cdot),S)-F(\bar x ,S))\|\tilde V^k\res B_1^d\times B_1^{n-d}\|\Big| (B_{1/2}^d\times B_1^{n-d})\\
&\, \,\leq \sup_{z\in B_{\bar x,2}}\Big| (F(\bar x +r_kz,S)-F(\bar x ,S))\Big|\|\tilde V^k\| (B_1^d\times B_1^{n-d})\\
&\overset{\eqref{eq:ok22}}{\leq} 2\Theta (\bar x,V)\|F\|_{C^1(B_{\bar x,2})}r_k\longrightarrow 0.
\end{split}
\end{equation}
Plugging \eqref{eq:limiteforte2} in \eqref{eq:limiteforte1}, we conclude by triangular inequality that
\begin{equation}\label{eq:limiteforte3}
\Big| (\Pi_S)_{\#}(F(\bar x,S)\|\tilde V^k\res B_1^d\times B_1^{n-d}\|)- \gamma \mathcal H^d\res B^d_1\Big| (B_{1/2}^d )\longrightarrow 0. 
\end{equation}
Since $\tilde V^k\res B_1^d\times B_1^{n-d}$ is an integral varifold, then $(\Pi_S)_{\#}\|\tilde V^k\res B_1^d\times B_1^{n-d}\|$ is a $d$-rectifiable measure in $\R^d \approx S$ with integer $d$-density $\theta_k(\cdot)\in L^1(B_1^d;\N;\mathcal L^d)$. 
By \eqref{eq:limiteforte3}, we deduce that 
$$F(\bar x,S) \theta_k(\cdot)\longrightarrow \gamma(\cdot) \qquad \mbox{in }L^1(B_{1/2}^d;\mathcal L^d)$$
 and consequently, up to subsequences, $(\theta_k(x))_k\subset \N$ converges for $\H^d$-a.e. $x \in B_{1/2}^d$ to $\frac{\gamma(x)}{F(\bar x,S)} \in \N$.  
By \eqref{eq:ok1}, we also know that
$$(\Pi_S)_{\#}\|\tilde V^k\res B_1^d\times B_1^{n-d}\| \rightharpoonup \Theta (\bar x,V) \mathcal L^d\res B_1^d
$$
and by uniqueness of the limit, we infer that $\frac{\gamma(\cdot)}{F(\bar x,S)} \equiv \Theta (\bar x,V)$ in $B_{1/2}^d$. 
But $\frac{\gamma(\cdot)}{F(\bar x,S)}$ is integer valued in $B_{1/2}^d$, so we conclude that $\Theta (\bar x,V) \in \N$ and that $V$ is an integral varifold.
\end{proof}

\begin{remark}
We recall that the isotropic version of Theorem \ref{integrality} above has been proved in \cite[Section 6.4]{Allard}, without the density assumption \eqref{densita}, which is a consequence of the monotonicity formula in the isotropic setting. If we were able to preserve in the limit varifold $V$ the lower density bound of the sequence $V_j$ of Theorem \ref{integrality}, we would get the full anisotropic counterpart of the compactness for integral varifolds and in particular an alternative proof of it in the isotropic setting. 
\end{remark}

\section{Proof of Theorem \ref{thm generale}}\label{mainresult}
Up to extracting subsequences, we can assume the existence of $V\in \mathbf V_d$ such that
\begin{equation}\label{muj va a mu}
  V_j \overset{*}{\rightharpoonup} V .
\end{equation}
We remark that condition $(a)$ of Theorem \ref{thm generale} is automatically satisfied by the lower semicontinuity of the functional $\FF(\cdot)$ with respect to varifolds convergence.
This implies straightforwardly also condition $(b)$.
For the remaining properties, we  divide the argument in several steps.

\subsection{Proof of Theorem \ref{thm generale}: stationarity of the limiting varifold}
In this section we prove condition $(c)$. 

Assume by contradiction that there exists a smooth vector field $\psi$  compactly supported in $\R^n\setminus H$ such that
$$\delta_F V(\psi)\leq -2C <0.$$
There exists a map $\varphi:t\in \R \mapsto \varphi_t\in C^\infty (\R^n;\R^n)$ solving the following ODE:
\begin{equation*}
\begin{cases}\frac{\partial \varphi_t(x)}{\partial t}=\psi(\varphi_t(x))&\quad \forall x\in \R^n,\\
\varphi_0(x)=x &\quad \forall x\in \R^n.
\end{cases}
\end{equation*}
Notice that one can choose an $\e>0$ small enough to have that $\varphi_t$ is actually a diffeomorphism of $\R^n\setminus H$ into itself for every $t\in [0,\e]$.
We set
$$ V_t:=(\varphi_t)_\# V, \qquad \mbox{and } V^j_t=(\varphi_t)_\# V_j.$$
By continuity of the functional $\delta_F Z(\psi)$ with respect to $Z$, up to take a smaller $\e>0$, we get that
$$\delta_F V_t(\psi)\leq -C <0, \quad \forall t\in [0, \e].$$
Integrating the last inequality, we conclude that
\begin{equation}
\label{integro1}\FF(V_\e)\leq \FF(V) - C\e.
\end{equation}
We fix an $R>0$ big enough so that  $\psi$ is supported in $B_R$ and fix $\alpha \in (1,2)$ such that 
\begin{equation}\label{bordo}
\|V_\e\|(\partial B_{\alpha R})=0, \quad \mbox{ and consequently }\quad \FF(V_\e,\partial B_{\alpha R})=0.
\end{equation}
We notice that equation \eqref{bordo} can be read as
\begin{equation}\label{bordo2}
\FF(V,\partial B_{\alpha R})=0, \quad \mbox{ because \quad $\varphi_\e=\Id$ in $B_R^c$}.
\end{equation}
Since $V^j_\e \rightharpoonup  V_\e$ and $V^j \rightharpoonup  V$, thanks to the equalities \eqref{bordo} and \eqref{bordo2}, one can infer 
\begin{equation}\label{limite}
\FF(V_\e,\overline{B_{\alpha R}})=\lim_j \FF(V^j_\e,\overline{B_{\alpha R}}), \quad \mbox{ and }\quad \FF(V,\overline{B_{\alpha R}})=\lim_j \FF(V^j,\overline{B_{\alpha R}}).
\end{equation}
Moreover, from \eqref{integro1} and the fact that $\varphi_\e=\Id$ in $B_R^c$, we also get
\begin{equation}
\label{integro2}\FF(V_\e,\overline{B_{\alpha R}})\leq \FF(V,\overline{B_{\alpha R}}) - C\e.
\end{equation}

Using \eqref{limite} and \eqref{integro2}, we infer
\begin{equation}\label{liminf}
\begin{split}
\liminf_j \FF(V^j_\e) &\, \,  =  \liminf_j (\FF(V^j_\e,\overline{B_{\alpha R}})+ \FF(V^j_\e,\overline{B_{\alpha R}}^c))\\
&\, \, \leq \limsup_j \FF(V^j_\e,\overline{B_{\alpha R}})+ \liminf_j\FF(V^j_\e,\overline{B_{\alpha R}}^c) \\
&\overset{\eqref{limite}}{=} \FF(V_\e,\overline{B_{\alpha R}})+ \liminf_j\FF(V^j_\e,\overline{B_{\alpha R}}^c) \\
&\overset{\eqref{integro2}}{\leq} \FF(V,\overline{B_{\alpha R}})- C\e+ \liminf_j\FF(V^j_\e,\overline{B_{\alpha R}}^c) \\
&\overset{\eqref{limite}}{=} \lim_j\FF(V^j,\overline{B_{\alpha R}})- C\e+ \liminf_j\FF(V^j_\e,\overline{B_{\alpha R}}^c) \\
&\, \, \leq \liminf_j\FF(V^j)- C\e.
\end{split}
\end{equation}
By definition of good class, see Definition \ref{def good class}, there exists a new sequence $(\tilde V_j)_{j\in \N}\subseteq \P(H)$, such that
$$\FF(\tilde V_j)\leq \FF(V^j_\e) + \frac{C\e}{4},$$
and passing to the lower limit on $j$, we get
$$\liminf_j \FF(\tilde V_j)\leq \liminf_j\FF(V^j)- \frac{3C\e}{4},$$
which contradicts the minimality of the sequence $(V_j)_{j\in \N}$.

\subsection{Proof of Theorem \ref{thm generale}: lower density estimates} 
In this section we show that if there exists $\delta>0$ such that
$$\Theta^d(x,V_j)\geq \delta, \qquad \mbox{for } \|V_j\|\mbox{-a.e. } x\in \R^n\setminus H, \quad \forall j\in \N,$$
then there exist $\theta_0=\theta_0(n,d,\delta,\lambda,\Lambda)>0$ such that
\begin{equation}
  \label{lower density estimate mu ball}
 \|V\|(B_{x,r})\ge\theta_0\,\omega_d r^d\,,\qquad \textrm{ \(x\in\spt\,\|V\|\) and \( r<d_x :=\dist(x,H)\)}.
\end{equation}
To this end, by \eqref{cost per area}, it is  sufficient to prove the existence of $\beta=\beta(n,d,\delta,\lambda,\Lambda)>0$ such that
\begin{equation*}
  \FF(V,Q_{x,l})\ge\beta\, l^d\,,\qquad \textrm{ \(x\in\spt\,\|V\|\) and \( l<2d_x/\sqrt{n}\)}\, .
\end{equation*}
Let us  assume by contradiction that there exist $x\in\spt\,\|V\|$ and $l<2d_x/\sqrt{n}$ such that 
\begin{equation*}
 \frac{\FF(V,Q_{x,l})^\frac{1}{d}}{l}<\b.
 \end{equation*}
We claim that this assumption, for $\beta$ chosen sufficiently small depending only on $n,d,\delta,\lambda$ and $\Lambda$, implies that for some $l_\infty \in (0,l)$
\begin{equation}
 \label{absurd lower bound}
   \FF(V,Q_{x,l_\infty})=0,
\end{equation}
which is a contradiction with $x \in \spt\,\|V\|$.
In order to prove \eqref{absurd lower bound}, we assume that $\FF(V,\pa Q_{x,l})=0$, which is true for a.e. $l\in \R_{>0}$. 

To prove \eqref{absurd lower bound}, we construct a sequence of nested cubes $Q_i := Q_{x,l_i}$ such that, if $\beta$ is sufficiently small, the following holds:
\begin{itemize}
 \item[(i)] $Q_0 = Q_{x,l}$;
 \item[(ii)] $\FF(V,\pa Q_i)=0$;
 \item[(iii)] setting  $m_i:=\FF(V,Q_i)$ then:
 $$
 \frac{m_i^{\frac{1}{d}}}{l_i}<\beta;
 $$
 \item[(iv)] $m_{i+1}\leq(1-\frac{1}{k_2})m_i$, where $k_2:=\frac{\Lambda k_1}{\lambda}$ and $k_1$ is the constant in Theorem \ref{deformationcube};
 \item[(v)] $(1-4\e_i)l_i\geq l_{i+1}\geq (1-6\e_i)l_i$, where 
 \begin{equation}\label{eq:epsi}
 \e_i:=\frac{1}{k\beta}\frac{m_i^{\frac{1}{d}}}{l_i}
 \end{equation}  
and  \(k=\max\{6,6 / (1-(\frac{k_2-1}{k_2})^\frac{1}{d})\}\) is a universal  constant. 
 \item [(vi)]   $\lim_i m_i =0$ and $\lim_i l_i >0$.
 \end{itemize}
Following \cite{DavidSemmes}, we are going to construct the sequence of cubes by induction: the cube $Q_0$ satisfies by construction hypotheses (i)-(iii). Suppose that cubes until step $i$ are already defined. 

Setting  $m_i^j:=\FF(V_j,Q_i)$, we cover $Q_i$ with the family $\L_{\e_i l_i}(Q_i)$ of closed cubes 
with edge length $\e_i l_i$ as described in Section \ref{notation} and we set \(C_1^i\) and \(C_2^i\) for the corresponding sets defined in \eqref{cornici}. We define $Q_{i+1}$ to be the internal cube given by 
the construction, and we note that  $C^i_2$ and $Q_{i+1}$ are non-empty if, for instance,
\[ 
   \e_i =  \frac{1}{k\beta}\frac{m_i^{\frac{1}{d}}}{l_i} <\frac{1}{k}\le\frac{1}{6} ,
\]
which is guaranteed by our choice of $k$. Observe moreover that $C^i_1\cup C^i_2$ is a strip of 
width at most $2\e_i l_i$ around $\pa Q_i$, hence the side $l_{i+1}$ of $Q_{i+1}$ satisfies $(1-4\e_i)l_i\leq l_{i+1} <(1-2\e_i) l_i $. 

We denote with $K_j$ the concentration set of $V_j$ (that is $ V_j:=\theta_j\H^d\res K_j\otimes \delta_{T_x{K_j}}$), where $\theta_j \in L^1(K_j;[\delta,+\infty);\H^d)$) and apply Theorem \ref{deformationcube} to $Q_i$, $V_j^\delta:=\frac 1\delta V_j$ and $\e=\e_i l_i$, obtaining the map $\Phi_{i,j}= \Phi_{\e_i l_i,V_j^\delta}$.
Notice that we are in the hypotheses to apply Theorem \ref{deformationcube}, since the rescaled varifolds $V_j^\delta$ have density bigger or equal than one in $Q_i$, $K_j$ is a relatively closed subset of $\R^n\setminus H$ and $Q_0\cap  H=\emptyset$.

 We claim that, for every $j$ sufficiently large,
\begin{equation} \label{conditionji}
  m_i^j \leq k_2 (m_i^j-m_{i+1}^j) + o_j(1).
\end{equation}
Indeed, since $(V_j)_{j\in \N}$ is a minimizing sequence in the class $\P(H)$, then $(V_j^\delta)$ is a minimizing sequence in the class
$$\P_\delta (H):=\left \{\frac 1\delta W \, : \, W\in \P(H)\right \}.$$
Since we are just rescaling the density of the varifolds and $\P(H)$ is a good class, also $\P_\delta (H)$ is a good class and by \eqref{cost per area}, we have that
\begin{equation*}
\begin{split}
 \frac{1}{\delta} m_i^j = \FF(V_j^\delta,Q_i)&\leq   \frac{1}{\delta} m_i + o_j(1)\leq \Lambda \|(\Phi_{i,j})_\#V_j^\delta\|(Q_i) +o_j(1) \\
     &=  \Lambda \|(\Phi_{i,j})_\#V_j^\delta\|( Q_{i+1} ) + \Lambda \|(\Phi_{i,j})_\#V_j^\delta\| \left(C^i_1\cup C^i_2 \right) +o_j(1)\\
        &\leq  \Lambda \|(\Phi_{i,j})_\#V_j^\delta\| \left(C^i_1\cup C^i_2 \right) +o_j(1) \leq \frac{k_2}{\delta} (m_i^j-m_{i+1}^j) + o_j(1).
        \end{split}
\end{equation*}
The last inequality holds because $\|(\Phi_{i,j})_\#V_j^\delta\|( Q_{i+1} )=0$ for $j$ large enough: otherwise, by property (6) of Theorem \ref{deformationcube}, there would exist $T \in \L^*_{\e_i l_i,d}(Q_{i+1})$ such that $\|(\Phi_{i,j})_\#V_j^\delta\|(T)\geq \H^d(T)$. 
Together with property (ii) and by \eqref{cost per area}, this would imply
\begin{equation*}
l_i^d\e_i^d = \H^d(T) \leq \|(\Phi_{i,j})_\#V_j^\delta\|(T) \leq \frac{k_1}{\delta} \|V_j\|( Q_i )\leq \frac{k_1}{\delta \lambda} m_i^j \rightarrow \frac{k_1}{\delta \lambda} m_i
\end{equation*}
and therefore, substituting \eqref{eq:epsi},
$$ \frac{m_i}{ k^d\beta^d} \leq \frac{k_1}{\delta \lambda} m_i,$$
which is false if $ \beta$ is sufficiently small ($m_i>0$ because $x\in\spt(\|V\|)$). 
Passing to the limit in $j$ in \eqref{conditionji} we obtain (iv): 
\begin{equation}\label{eq:stimai}
m_{i+1}\leq \frac{k_2-1}{k_2} m_i. 
\end{equation}
Since $l_{i+1}\geq(1-4\e_i)l_i$, we can slightly shrink the cube $Q_{i+1}$ to a concentric cube $Q'_{i+1}$ 
with $l'_{i+1}\geq(1-6\e_i) l_i>0$, $\FF(V,\pa Q'_{i+1})=0$ and for which (iv) still holds, just getting a lower value for $m_{i+1}$. 
With a slight abuse of notation, we rename this last cube $Q'_{i+1}$ as $Q_{i+1}$.

We now show (iii). Using \eqref{eq:stimai} and condition (iii) for $Q_i$, we obtain
\begin{equation*}
\frac{m_{i+1}^\frac{1}{d}}{l_{i+1}} \leq \left(\frac{k_2-1}{k_2}\right)^\frac{1}{d}
\frac{m_i^\frac{1}{d}}{(1-6\e_i)l_i} < \left(\frac{k_2-1}{k_2}\right)^\frac{1}{d}
\frac{\beta}{1-6\e_i}.
\end{equation*}
The last quantity will be less than $\beta$ if
\begin{equation}\label{eq:i}
\left(\frac{k_2-1}{k_2}\right)^\frac{1}{d} \leq 1-6\e_i=1-\frac{6}{k \beta} \frac{m_i^\frac{1}{d}}{l_i} .
\end{equation}
In turn, inequality \eqref{eq:i} is true because (iii) holds for $Q_i$, provided we choose $k\geq 6 / \big(1-(1-1/k_2)^\frac{1}{d}\big)$. 
Furthermore, estimating $\e_0<1/k$ by (iii) and (v), we  also  have $\e_{i+1}\leq\e_i$.

We are left to prove (vi): $\lim_i m_i=0$ follows directly from (iv); regarding the non degeneracy of the cubes, note that 
\begin{equation*}
\begin{split}
\frac{l_\infty}{l_0}:=\liminf_i \frac{l_i}{l_0} &\geq  \prod_{i=0}^{\infty}(1-6\e_i)=\prod_{i=0}^{\infty}\left (1-\frac{6}{k \beta}\frac{m_i^\frac{1}{d}}{l_i}\right ) \\
&\geq
\prod_{i=0}^{\infty}\left (1-\frac{6m_0^\frac{1}{d}}{k \beta l_0 \prod_{h=0}^{i-1}(1-6\e_h)}\left(\frac{k_2-1}{k_2}\right)^\frac{i}{d} \right)\\
 &\geq\prod_{i=0}^{\infty}\left (1-\frac{6}{k(1-6\e_0)^i}\left(\frac{k_2-1}{k_2}\right)^\frac{i}{d}\right),
\end{split}
\end{equation*}
where we used  $\e_h\leq \e_0$ in the last inequality. Since $\e_0<1/k$, the last product is strictly positive, provided 
$$
k>\frac{6}{1-\left(\frac{k_2-1}{k_2}\right)^\frac{1}{d}},
$$
which is guaranteed by our choice of $k$. 
We conclude that $l_\infty>0$, which ensures claim \eqref{absurd lower bound}.

\subsection{Proof of Theorem \ref{thm generale}: rectifiability of the limiting varifold}   In this section, we prove condition $(d)$. Indeed, with the assumption on the uniform density lower bound of the minimizing sequence, by the previous step we know that $V$ enjoys the denisty lower bound \eqref{lower density estimate mu ball}. Moreover, by condition $(c)$, it is $F$-stationarity in $\R^n\setminus H$. Since $F$ is as in Definition \ref{atomica}, we are in the hypotheses to apply Theorem \ref{thm:rect} and to conclude that $V\res G(\R^n\setminus H)$ is a $d$-rectifiable varifold.

Moreover, by the previous step, for every $x \in \spt\, \|V\|\setminus H$ we have \eqref{lower density estimate mu ball}. It follows that
$$ \spt\, \|V\|\setminus H\subseteq \mbox{conc}(V) \subseteq \spt\, \|V\|.$$
We conclude that $\mbox{conc}(V)\setminus H = \spt\, \|V\|\setminus H$ and consequently that the concentration set is relatively closed in $\R^n\setminus H$.

\subsection{Proof of Theorem \ref{thm generale}: integrality of the limiting varifold}  In this section we prove that, under the further assumption that the minimizing sequence is made of integral varifolds satisfying 
$$\sup_j |\delta_FV_j|(W) <\infty, \qquad \forall \, W \subset \subset (\R^n \setminus H),$$
then $V\res G(\R^n\setminus H)$ is integral.
Indeed, we already know that $V$ enjoys the density lower bound \eqref{lower density estimate mu ball}.

We are in the hypotheses to apply Theorem \ref{integrality} with $U:=\R^n \setminus H$ and conclude that condition $(e)$ holds.


\begin{thebibliography}{10}

\bibitem{Allard}
W.~K. Allard.
\newblock {On the first variation of a varifold}.
\newblock {\em Ann. Math.}, 95:417--491, 1972.

\bibitem{Allardratio}
William~K. Allard.
\newblock A characterization of the area integrand.
\newblock In {\em Symposia {M}athematica, {V}ol. {XIV} ({C}onvegno di {T}eoria
  {G}eometrica dell'{I}ntegrazione e {V}ariet\`a {M}inimali, {INDAM}, {R}ome,
  1973)}, pages 429--444. Academic Press, London, 1974.

\bibitem{Allard84BOOK}
William~K. Allard and Frederick~J. Almgren, Jr., editors.
\newblock {\em Geometric measure theory and the calculus of variations},
  volume~44 of {\em Proceedings of Symposia in Pure Mathematics}. American
  Mathematical Society, Providence, RI, 1986.

\bibitem{AlmgrenINVITATION}
F.~J.~Jr. Almgren.
\newblock {\em {Plateaus' problem. An invitation to varifold geometry}}.
\newblock {Mathematics Monograph Series}. W. A. Benjamin, Inc., New
  York-Amsterdam, 1966.

\bibitem{Almgren68}
F.~J.~Jr Almgren.
\newblock {Existence and regularity almost everywhere of solutions to elliptic
  variational problems among surfaces of varying topological type and
  singularity structure}.
\newblock {\em Ann. Math.}, 87:321--391, 1968.

\bibitem{Almgren76}
F.~J.~Jr. Almgren.
\newblock {Existence and regularity almost everywhere of solutions to elliptic
  variational problems with constraints}.
\newblock {\em Mem. Amer. Math. Soc.}, 4(165):viii+199 pp, 1976.

\bibitem{AFP}
L.~Ambrosio, N.~Fusco, and D.~Pallara.
\newblock {\em {Functions of bounded variation and free discontinuity
  problems}}.
\newblock {Oxford Mathematical Monographs}. The Clarendon Press, Oxford
  University Press, New York, 2000.

\bibitem{DavidSemmes}
G.~David and S.~Semmes.
\newblock {Uniform rectifiability and quasiminimizing sets of arbitrary
  codimension.}
\newblock {\em Memoirs of the American Mathematical Society}, pages viii+132
  pp., 2000.

\bibitem{DeGiorgiSOFP1}
E.~{De Giorgi}.
\newblock {Su una teoria generale della misura $(r-1)$-dimensionale in uno
  spazio ad $r$-dimensioni}.
\newblock {\em Ann. Mat. Pura Appl. (4)}, 36:191--213, 1954.

\bibitem{DeLellisNOTES}
C.~{De Lellis}.
\newblock {\em {Rectifiable sets, densities and tangent measures}}.
\newblock {Zurich Lectures in Advanced Mathematics}. European Mathematical
  Society, Z{\"u}rich, 2008.

\bibitem{DeLDeRGhi16}
C.~De~Lellis, A.~De~Rosa, and F.~Ghiraldin.
\newblock A direct approach to the anisotropic {P}lateau's problem.
\newblock {\em Available on arxiv}, 2016.

\bibitem{DeLGhiMag}
C.~{De Lellis}, F.~Ghiraldin, and F.~Maggi.
\newblock {A direct approach to Plateau's problem}.
\newblock {\em JEMS}, 2014.

\bibitem{DePDeRGhi3}
G.~De~Philippis, A.~De~Rosa, and F.~Ghiraldin.
\newblock {Existence results for minimizers of parametric elliptic
  functionals}.
\newblock 2016.
\newblock {Forthcoming}.

\bibitem{DePDeRGhi2}
G.~De~Philippis, A.~De~Rosa, and F.~Ghiraldin.
\newblock {Rectifiability of varifolds with locally bounded first variation
  with respect to anisotropic surface energies}.
\newblock Accepted on {\em {Comm. Pure Appl. Math.}}, 2016.
\newblock  Available on arxiv

\bibitem{DePDeRGhi}
G.~De~Philippis, A.~De~Rosa, and F.~Ghiraldin.
\newblock A direct approach to {P}lateau's problem in any codimension.
\newblock {\em Adv. in Math.}, 288:59--80, January 2015.

\bibitem{FedererBOOK}
H.~Federer.
\newblock {\em {Geometric measure theory}}, volume 153 of {\em {Die Grundlehren
  der mathematischen Wissenschaften}}.
\newblock Springer-Verlag New York Inc., New York, 1969.

\bibitem{federer60}
H.~Federer and W.~H. Fleming.
\newblock {Normal and integral currents}.
\newblock {\em Ann. Math.}, 72:458--520, 1960.

\bibitem{HarrisonPugh14}
J.~Harrison and H.~Pugh.
\newblock {Existence and soap film regularity of solutions to {P}lateau's
  problem}.
\newblock {\em Advances in Calculus of Variations}, 9(4):357–394, 2015.

\bibitem{HarrisonPugh15}
J.~Harrison and H.~Pugh.
\newblock {Solutions To Lipschitz Variational Problems With Cohomological
  Spanning Conditions}.
\newblock 2015.
\newblock arxiv:1506.01692.

\bibitem{HarrisonPugh16}
J.~Harrison and H.~Pugh.
\newblock {General Methods of Elliptic Minimization}.
\newblock 2016.
\newblock arxiv: 1603.04492.

\bibitem{Preiss}
D.~Preiss.
\newblock {Geometry of measures in {${\bf R}\sp n$}: distribution,
  rectifiability, and densities}.
\newblock {\em Ann. of Math. (2)}, 125(3):537--643, 1987.

\bibitem{reifenberg1}
E.~R. Reifenberg.
\newblock {Solution of the {P}lateau problem for $m$-dimensional surfaces of
  varying topological type}.
\newblock {\em Acta Math.}, 104:1--92, 1960.

\bibitem{schoensimonalmgren}
R.~Schoen, L.~Simon, and F.~J.~Jr. Almgren.
\newblock {Regularity and singularity estimates on hypersurfaces minimizing
  parametric elliptic variational integrals. I, II.}
\newblock {\em Acta Math.}, 139(3-4):217--265, 1977.

\bibitem{SimonLN}
L.~Simon.
\newblock {\em {Lectures on geometric measure theory}}, volume~3 of {\em
  {Proceedings of the Centre for Mathematical Analysis}}.
\newblock Australian National University, Centre for Mathematical Analysis,
  Canberra, 1983.

\end{thebibliography}
\end{document}